\documentclass{amsart}

\usepackage{epsfig}
\usepackage{amsmath}
\usepackage{amssymb}
\usepackage{amscd}
\usepackage{graphicx}
\usepackage{color}
\usepackage{verbatim}
\usepackage{dsfont}
\usepackage[symbol*]{footmisc}

\DefineFNsymbols*{asterisks}{{$\star$}{$\dagger$}{$\ddagger$}{$\dagger\dagger$}{$\ddagger\ddagger$}}
\setfnsymbol{asterisks}

\newtheorem{thm}{Theorem}[section]
\newtheorem{lem}[thm]{Lemma}

\newtheorem{que}[thm]{Question}

\newtheorem{defn}[thm]{Definition}

\theoremstyle{remark}

\def \N {\mathbb N}

\def \G {\mathcal G}

\def \R {\mathbb R}

\def \P {\mathbb P}

\def \corr {\mathsf{corr}}
\def \sq {sequence}
\def \xt {$(X,T)$}

\def \tl {topological}

\def \ds {dynamical system}

\def \mmu {\mu}

\def \corr {\mathsf{corr}}

\numberwithin{equation}{section}

\title{Almost full entropy subshifts uncorrelated to the M\"obius function}
       
\author{Tomasz Downarowicz and Jacek Serafin}

\address{$^\dagger$ Faculty of Pure and Applied Mathematics, Wroc{\l}aw University of Science and Technology,
Wybrze\.ze Wyspia\'nskiego 27, Wroc{\l}aw 50-370, Poland, \rm downar@pwr.edu.pl, serafin@pwr.edu.pl}


\thanks{The research is supported by the NCN (National Science Center, Poland) grant 2013/08/A/ST1/00275.}

\begin{document}
\begin{abstract}
We show that if $y=(y_n)_{n\ge 1}$ is a bounded sequence with zero average along every infinite arithmetic progression then for every $N\ge 2$ there exist (unilateral or bilateral) subshifts $\Sigma$ over $N$ symbols, with entropy arbitrarily close to $\log N$, uncorrelated to $y$. In particular, for $y=\mmu$ being the M\"obius function, we get that there exist subshifts as above which satisfy the assertion of Sarnak's conjecture. The existence of positive entropy systems uncorrelated to the M\"obius function is claimed in Sarnak's survey \cite{sarnak} (and attributed to Bourgain), however, to our knowledge no examples have ever been published. We fill in this gap and by the way we show that this has nothing to do with more advanced algebraic properties (for instance multiplicativity) of the considered \sq.
\end{abstract}
\maketitle

\numberwithin{equation}{section}

\section{Introduction}
Let $y$ be a bounded, real-valued \sq\ with zero average along every infinite arithmetic progression, i.e., satisfying, for every $t\ge 1$ and $l\ge 0$, the condition 
\begin{equation}\label{arp}
\lim_n \frac1n\sum_{i=1}^n y_{it+l} = 0.
\end{equation}
Following the terminology used for multiplicative functions (see e.g. \cite{FH}), we will call any such \sq\ \emph{aperiodic}. Without loss of generality we will assume that $|y_n|\le 1$ for all $n$. For example, we can take the M\"obius function $y=\mmu$, where
$$
\mmu_n=\begin{cases}
\phantom{-}1&\text{for $n=1$,}\\
\phantom{-}(-1)^r&\text{if $n$ is a product of $r$ distinct primes,}\\
\phantom{-}0& \text{otherwise (i.e., if $n$ has a repeated prime factor).}
\end{cases}
$$
It is known that this \sq\ satisfies the condition \eqref{arp} (see e.g. \cite{sarnak}). 

Once an aperiodic \sq\ $y$ is fixed, we consider \tl\ \ds s \xt\ where $X$ is a compact metric space and $T:X\to X$ is a continuous transformation. Subshifts (in which the transformation is always the left shift) will be denoted using just one letter $\Sigma$. Uncorrelation between a system and a \sq\ will be understood as follows:

\begin{defn}\label{def1}
We say that \xt\ is \emph{uncorrelated} to $y$ if for each continuous function $f:X\to\R$ and every $x\in X$, we have 
$$
\lim_n\frac 1n\sum_{i=1}^n f(T^ix)y_i = 0.
$$
\end{defn}

The celebrated Sarnak's conjecture (\cite{sarnak}) asserts that any system with zero \tl\ entropy is uncorrelated to the  M\"obius function. Most of the activity around this conjecture is aimed toward determining ever larger classes of zero entropy systems which obey the Sarnak's uncorrelation condition 
(for a long list of references see the survey \cite{AKLR}, newer results are in \cite{HWZ}). Much less (but not zero) effort is devoted to finding systems which correlate with $\mmu$ 
(see e.g. \cite{AKL}, \cite{DK}, \cite{K}). Clearly, all examples found so far have positive entropy. 
For these efforts to be meaningful, it becomes crucial to also consider the ``inverse'' of Sarnak's 
problem i.e., the following question
\begin{que}
Are there positive entropy systems uncorrelated to $\mmu$? 
\end{que}
\noindent As a matter of fact, in Sarnak's exposition \cite{sarnak} it is claimed that such systems do exist, and relevant example is attributed to Bourgain. However, no examples have ever been published and we failed to acquire any details, so we decided to consider the question as open.
 
Notice that if the answer to the above question was negative and if Sarnak's conjecture held, one could view $\mmu$ as a \sq\ able to precisely differentiate between positive and zero \tl\ entropy systems. As we will show, it is not the case.

In this work we answer the above question in the positive, providing evidence for the claim in Sarnak's survey. In fact we show a bit more: if $y$ is any aperiodic \sq\ as described at the beginning of this section, and $N\ge 2$ is an arbitrary integer, then there exist subshifts $\Sigma$ on $N$ symbols, with entropy arbitrarily close to $\log N$, uncorrelated to $y$. The proof relies on a complicated counting blocks argument.
\medskip

As a byproduct\footnote{We are sure that this remark has a more direct proof.} we can make the following remark. Weiss \cite[Theorem 8.3]{W} proved that any subshift $\Sigma$ on $N$ symbols, of entropy larger than $\log(N-1)$, has a positive density independence set $A$ (i.e., a subset of $\mathbb N$ along which all combinations of symbols occur), with a lower bound on the density $\mathsf{dens}(A)$ in $\mathbb N$ depending (obviously nondecreasingly) on the entropy $h(\Sigma)$. Since $h(\Sigma) = \log N$ holds only for the full shift (whose independence set is the whole $\mathbb N$), one might expect, that as $h(\Sigma)$ tends to $\log N$, these lower bounds tend to $1$. Our examples show that the limit of these bounds cannot exceed $\frac12$. Indeed, take any aperiodic sequence $y$ over $\{-1,1\}$ (in fact almost every \sq\ is such, for the $(\frac12,\frac12)$-Bernoulli measure). In our example created for $y$, with $\Lambda=\{1,2,\dots,N\}$ and entropy arbitrarily close to $\log N$, consider the function $f(x)=(-1)^{x_0}$. There are points $x$ such that $f(x)$ matches $y$ along the independence set $A$. Then $(f(T^nx))_{n\ge 1}$ correlates with $y$ by at least $2\,\mathsf{dens}(A)-1$, so this number cannot be positive.

\section{Preliminaries}
Let $B=(b_1,\dots,b_n)\in \R^n$ be a finite \sq\ (block) of real numbers. We define its \emph{average} as
$$
\overline B =\frac1n\sum_{i=1}^nb_i.
$$
If $C=(c_1,\dots,c_n)$ is another block (of the same length) we define
$$
\corr(B,C) = |\overline{BC}|,\footnote{Formally, the correlation should be the distance between the average of the product and the product of the averages. But since our aperiodic \sq\ has zero average, we can skip the latter term.}
$$
where $BC = (b_1c_1,\dots,b_nc_n)$. Further, if $x=(x_i)_{i\geq 1}$ and $y=(y_i)_{i\geq 1}$ are bounded sequences, the correlation between $y$ and $x$ is defined as 
$$
\corr(x, y)=\limsup_{n\to\infty} \corr(x_1^n, y_1^n),
$$
where $z_m^n$ stands for the block $(z_m,\dots,z_n)$ ($m\le n$).
\medskip

An elementary lemma concerns \sq s with zero average:

\begin{lem}\label{proportion}
Let $(y_n)_{n\ge 1}$ be a bounded \sq\ with zero average. Then, for every $\epsilon\in(0,1)$ and every natural $m$ there exists a natural $L(\epsilon,m)$ such that for every $L\ge L(\epsilon,m)$ the absolute value of the average of $y$, over any interval $I\subset[1,mL]$ of length at least $L$, is less than $\epsilon$.
\end{lem}
\begin{proof} Without loss of generality, we can assume that $|y_n|\le 1$ for all $n$.
Now simply define $L(\epsilon,m)$ to be such that for every $n\ge\frac\epsilon2L(\epsilon,m)$
the average of $y$ over $[1,n]$ is less than $\frac\epsilon{2m}$ in absolute value. Consider an interval $I$ of some length $i\ge L\ge L(\epsilon,m)$ as in the assertion of the lemma, and denote by $J$ the interval extending from $1$ to the left end of $I$ and let $j$ be its length. Note that $\frac{i+j}i\le m$. Denote by $\alpha,\beta$ and $\gamma$ the averages of $y$ over $I$, $J$ and $I\cup J$, respectively. We have
$$
\gamma = \frac i{i+j}\alpha + \frac j{i+j}\beta,
$$
hence
$$
|\alpha|\le\frac{i+j}i|\gamma| + \frac ji|\beta|\le m|\gamma| + \frac ji|\beta|.
$$
Since $i+j\ge L(\epsilon,m)>\frac\epsilon2L(\epsilon,m)$, we have $|\gamma|<\frac\epsilon{2m}$. If $\frac ji<\frac\epsilon2$, we are done (because $|\beta|\le 1$). Otherwise $j\ge \frac\epsilon2i\ge \frac\epsilon2L(\epsilon,m)$, hence $|\beta|<\frac\epsilon{2m}$ while $\frac ji\le m$, and we are done as well.

\end{proof}

We will need a subtle version of Hoeffding's inequality \cite[Theorem~3]{H63}, which we formulate in the form that suits us best:

\begin{thm}
Let $\mathsf X_1, \mathsf X_2, \ldots, \mathsf X_m$ be independent (not necessarily identically distributed) random variables, each taking values in the interval $[-1,1]$ and with a common bound
$\mathbf v$ on the variance. Then for every $\epsilon>0$ the following inequality holds:
$$
\P\{\overline{\mathsf X}-E\overline{\mathsf X}\geq \epsilon\}\leq \left[(1+\tfrac{2\epsilon}{\mathbf v})^{\frac{\mathbf v+2\epsilon}{\mathbf v+4}}\cdot(1-\tfrac\epsilon2)^{(1-\frac\epsilon2)\frac4{\mathbf v+4}}\right]^{-m},
$$
where $\overline{\mathsf X}=\frac1m\sum_{i=1}^m\mathsf X_i$ and $E\overline{\mathsf X}$ stands for the expectation of  $\overline{\mathsf X}$. 
\end{thm}
Of course, in order to obtain a more convenient upper bound on the above probability, 
we can replace the expression in square brackets (henceforth denoted by $W$) by a smaller (yet positive) one. 
First, since in our case it only makes sense to consider $\epsilon<2$, we can use the facts that $\frac{\mathbf v+2\epsilon}{\mathbf v+4}\geq \frac{\epsilon}{2}$ and that $x^x>\frac12$ for $x\in(0,1)$, and write
$$
W \ge (1+\tfrac{2\epsilon}{\mathbf v})^{\frac{\epsilon}{2}}\cdot(\tfrac12)^
{\frac4{\mathbf v+4}} = W_1.
$$
Now we can simply skip the first $1$ and the second (smaller than $1$) exponent:
$$
W_1\ge (\tfrac{2\epsilon}{\mathbf v})^{\frac{\epsilon}{2}}\cdot\tfrac12 = \tfrac12(2\epsilon)^{\frac\epsilon2}{\mathbf v}^{-\frac\epsilon2} = W_2.
$$
Finally, we note that $(2x)^{\frac x2}>\frac12$ for $x\in(0,1)$, hence
$$
W_2>\tfrac14{\mathbf v}^{-\frac\epsilon2}=W_3.
$$
Replacing $W$ by $W_3$ in Hoeffding's inequality and combining with a symmetric estimate for $-\overline{\mathsf X}$, we obtain
\begin{equation}\label{Hoeffding}
\P\{|\overline{\mathsf X}-E\overline{\mathsf X}|\geq \epsilon\}< 2\cdot4^m{\mathbf v}^{\frac\epsilon2m}.
\end{equation}

\section{The main result}

\begin{thm}\label{thm}
Let $y$ be an aperiodic sequence and $N\ge 2$ be a fixed integer. 
There exists a subshift $\Sigma$ over $N$ symbols of entropy arbitrarily close to $\log N$, uncorrelated to $y$.
\end{thm}

\begin{proof}
We need to guarantee uncorrelation to $y$ of any \sq\ obtained using a function $f\in C(\Sigma)$ (and a starting point $x$). Since uncorrelation to $y$ is preserved under linear combinations and uniform limits of \sq s, it suffices to consider functions $f$ from a family linearly dense in $C(\Sigma)$. We can choose in this role the family consisting of $\{-1,1\}$-valued functions depending on finitely many coordinates. Indeed, the collection of linear combinations of such functions is an algebra which contains constants and separates points of $\Sigma$, and thus the Stone--Weierstrass Theorem applies. Further, even for bilateral subshifts, it suffices to consider functions which depend on finitely many \emph{nonnegative} coordinates. Indeed, if a function depends also on some negative coordinates, composing it with an appropriate iterate of the shift we obtain a function depending only on nonnegative coordinates, and which yields the same set of values of the correlation with $y$.

Any $\{-1,1\}$-valued function $f$ depending on finitely many nonnegative coordinates will be called a \emph{code} and the \emph{horizon} of $f$ (denoted by $r_{\!f}$) is defined as the minimal $r\ge 1$ such that $f$ does not depend on the coordinates $r+1,r+2,\dots$ (for $f$ constant we have $r_{\!f}=1$). The name ``code'' is justified by the fact that every such $f$ determines a \emph{sliding block code} which can be applied to any block $B$ appearing in $\Sigma$ of any length $n$ larger than or equal to $r_{\!f}$, producing a block $f(B)$ over $\{-1,1\}$, of length $n-r_{\!f}+1$. The rule is 
$$
f(B)_i = f(b_i,b_{i+1},\dots,b_{i+r_{\!f}-1}).
$$
Since there are countably many such codes, we enumerate them by natural numbers (to be used later).

\medskip
Our goal is to build a special subshift $\Sigma$. This subshift will be the intersection of a nested \sq\ of subshifts $\Sigma_k$, where each $\Sigma_k$ consists of all (unilateral or bilateral, depending of the preferred type of the subshift) infinite concatenations (and their shifts) of blocks belonging to some family $\G_k\subset\Lambda^{N_{\!k}}$, where $\Lambda$ is an alphabet of cardinality $N$ (i.e., $\G_k$ is a subfamily of blocks of some common length $N_{\!k}$). It is an elementary exercise to show that the \tl\ entropy of the intersection of a nested \sq\ of subshifts equals the limit of their entropies\footnote{Attention, this is not true for more general \tl\ \ds s.}, thus
$$
h(\Sigma)=\lim_kh(\Sigma_k) = \lim_k\tfrac1{N_{\!k}}\log(\#\G_k).
$$
 
We begin the construction by setting $N_0=1$ and $\G_0=\Lambda$. Now $\Sigma_0$ is simply the full shift on $N$ symbols.
In each following step $k\ge 1$ of the construction we will refer to several parameters, for which we are about to fix the notation consistently used throughout the remainder of this paper. In the description of the inductive step $k$, in the notation of most of these parameters we will skip the subscript $k$. And so: 
\begin{itemize}
	\item The \emph{multiplier} $m=m_k$ will play the role of the ratio $\frac{N_{\!k}}{N_{\!k-1}}$; the 
	family $\G_k$ will consist of concatenations of $m$ blocks from $\G_{k-1}$. In the first step $m$ is 
	equal to some large $M\geq 81$, then it tends nondecreasingly to infinity, but very slowly (it is constant on 
	long intervals, rarely increasing by $1$). The dependence $k\mapsto m$ will be specified more precisely 	
	later. 
	\item The length $N_{\!k}$ equals the product $m_1m_2\cdots m_k$. Clearly, $M^k\le N_{\!k}\le 
	m^k$. Since in most formulas this parameter appears several times with different indices, exceptionally,
	we will never skip the subscript (besides, $N$ is already reserved to denote the cardinality of the alphabet).
\end{itemize}
All of the following parameters depend on $k$ indirectly, via the multiplier $m$.	
\begin{itemize}
	\item We fix a decreasing to zero \sq\ of parameters $\epsilon=\epsilon_m$ starting with $\epsilon_M=1$ 
	and assuming the values $\epsilon_m=\frac3m$ for $m>M$. Clearly, $\epsilon$ tends to zero with $k$, 
	but very slowly, remaining constant throughout many steps. 
	\item The number $p=m-M$ (which is always strictly less than $k$) will serve as the index of 
	some previous step (called the \emph{reference step}); we will view the elements of $\G_k$ (and also of 
	$\G_{k-1}$) as concatenations of the blocks from $\G_p$. In the initial step we have $m=M$, so that $p=0$ 
	and we imagine the elements of $\G_k$ decomposed into elements of $\G_0$ (single symbols). 
	\item We will also refer to the multiplier that was used in step number $p+1$. According to our 
	notation it is $m_{p+1}$. The number $2^{-{m_{p+1}}}$ will be denoted by $\delta$. This parameter tends to 
	zero with $k$ (but very very slowly).
	\item Our estimates in step $k$ will involve some finite collection $\mathcal F=\mathcal F_k$ of codes. 
	Two	conditions must be fulfilled to include a code $f$ in this collection: its index in the ordering of 
	all codes must not exceed $m$, and its horizon $r_{\!f}$ must not exceed $\delta N_{\!p}$. It is clear 
	that $\#\mathcal F\le m$. The numbers $\delta N_{\!p}=N_{\!p}2^{-m_{p+1}}\ge 
	N_{\!p}2^{-p-M}\ge2^{-M}(\frac M2)^p$ tend nondecreasingly to infinity, hence the collections ascend, and 
	every code will eventually be included.
\end{itemize}
\medskip

We can now define $\G_k$ more precisely: Suppose that for some $k\ge 1$ the family $\G_{k-1}\subset\Lambda^{N_{\!k-1}}$ has been established. We define $\G_k$ as the family of all concatenations $B$ of $m$ blocks from $\G_{k-1}$ which satisfy the following requirement: 
\begin{enumerate}
\item[(R)]\label{R} for every $1\le j\le (m^2-1)N_{\!k}$ and every $f\in\mathcal F$, letting $C=y_j^{j+N_{\!k}-1}$ 
we have $|\overline{f(B)C}|<2(\epsilon+\delta)$. \label{dwa}
\end{enumerate}
(by convention, we denote by $\overline{f(B)C}$ what should formally be $\overline{f(B)C'}$, where
$C'$ is $C$ trimmed by $r_{\!f}-1$ terminal symbols, to match the length of $f(B)$). 
In words, we require that all images of $B$ under the codes from $\mathcal F$ have small correlations with every block of $y$ of length $N_{\!k}$, ending before the position $m^2N_{\!k}$. 

\medskip
We can identify the family of all concatenations of $m$ blocks from $\G_{k-1}$ with the product space $(\G_{k-1})^m$. Notice that if $\G_{k-1}$ is equipped with the normalized counting measure, then the product measure on $(\G_{k-1})^m$ coincides with the normalized counting measure. Similarly, this measure conditioned on $\G_k$ is the normalized counting measure. In order to estimate (from below) the cardinality of $\G_k$, we need to estimate the probability $\gamma_k$ that a block $B\in(\G_{k-1})^m$ satisfies (R). Then
$$
\#\G_k = (\#\G_{k-1})^m\gamma_k.
$$
Recursive application of the above dependence (in which we replace the varying parameter $m$ by $\tfrac{N_k}{N_{k-1}}$) yields
$$
\#\G_k = N^{N_{\!k}}\cdot\gamma_1^{\frac{N_{\!k}}{N_1}}\cdot\gamma_2^{\frac{N_{\!k}}{N_2}}\cdots
\gamma_{k-1}^{\frac{N_{\!k}}{N_{\!k-1}}}\cdot\gamma_k^{\frac{N_{\!k}}{N_{\!k}}}.
$$
This, and the convergence of entropies, allows us to write the entropy of $\Sigma$ as
$$
h(\Sigma)= \lim_k\frac1{N_{\!k}}\log(\#\G_k)= \log N + \sum_{k=1}^\infty\frac{\log(\gamma_k)}{N_{\!k}}. 
$$

If we arrange (which we will) that all $\gamma_k$ are larger than or equal to $\frac12$, then we shall have
$$
h(\Sigma)\ge\log N -\log 2 \sum_{k=1}^\infty\frac1{M^k}= \log N-\frac1{M-1}\log2. 
$$
So, solely by the choice of the initial multiplier $M$, we will be able make $h(\Sigma)$ as close to $\log N$ as we wish.
\medskip

Now we specify the assignment $k\mapsto m$. For each $m\ge M$ let the \emph{jump step $K_m$} be defined as the first index $k$ such that $m_k=m$. Clearly, $K_{\!M}=1$, so there is no choice, but for $m>M$ we are free to choose the jump steps arbitrarily large. We choose them so that they satisfy two requirements:
\begin{enumerate}
\item[(a)] In step $K_m$ the reference index $p$ will be increased from $m-1-M$ to $m-M$. Recall that $N_{\!p}$ is the length of the blocks built in step $p$. This parameter will not change regardless of how we choose $K_m$. We require that the ratio $\frac{N_{\!K_m}}{N_{p}}$ is at least as large as the maximum of the parameters	$L(\epsilon,m^2)$ evaluated for $y$ along all arithmetic progressions of the form $(iN_{\!p}+l)_{i\ge 1}$ with $0\le l<N_{\!p}$. Recall that $y$ has zero average along every arithmetic progression, so Lemma \ref{proportion} applies. 
\item[(b)] Let $\alpha(m)=m^4\cdot2\cdot 4^m(2N_{\!p})^{\frac32}$. We require $K_m$ to be so large that\hfill\break $9\cdot\alpha(m)\cdot(\frac89)^{K_m-1}<\frac1{2^{m+2}}$.
\end{enumerate}
This concludes the construction of the subshift. Now we need to prove its properties. 
\medskip

In step $k$ let us fix a block $C$ and a code $f$, as they appear in the condition (R).
On the probability space $(\G_{k-1})^m$ let us denote by $\mathsf X$ the random variable
$B\mapsto\overline{f(B)C}$. 

\begin{lem} With the above notation, for every $k$, $f$ and $C$, we have:
\begin{enumerate}
  \item[(A)] $\sum_{s=p+1}^{k-1}(1-\gamma_s)<\frac{\delta}2$.
  \item[(B)] $|E\mathsf X|<\epsilon+2\delta$,
	\item[(C)] $\gamma_k>1-\alpha(m)(\tfrac89)^{k-1}$, (which, by {\rm(b)} and since $k\ge K_m$, is much larger than $\frac12$).
\end{enumerate}
\end{lem}

\begin{proof}

In steps $1,2,\dots K_{M+1}-1$, the multiplier equals $M$, the reference index is $0$, hence $\epsilon=1$ implying that $\gamma_k=1$ and all three conditions hold trivially.

Fix any $k$ such that the corresponding multiplier $m$ is larger than $M$. We have $k\ge K_m$. 
Suppose we have proved the lemma for all indices smaller than $k$. Then~(A) holds. Indeed, using~(C) for indices smaller than $k$, and (b), we can compute as follows:
\begin{multline*}
\sum_{s=p+1}^{k-1}(1-\gamma_s)\le \sum_{n\ge m_{p+1}}\sum_{s=K_n}^{K_{n+1}} \alpha(n)(\tfrac89)^{s-1}\le\\
\sum_{n\ge m_{p+1}} \alpha(n)\cdot9\cdot(\tfrac89)^{K_{n}-1}\le\sum_{n\ge m_{p+1}}\frac1{2^{n+2}}
=\frac1{2^{m_{p+1}+1}}=\frac{\delta}2.
\end{multline*}

We pass to proving~(B). We have the following descending sets:
$$
(\G_p)^{\frac{N_k}{N_p}}\supset (\G_{p+1})^{\frac{N_k}{N_{p+1}}}\supset\cdots\supset (\G_{k-2})^{\frac{N_k}{N_{k-2}}}\supset (\G_{k-1})^m,
$$
each containing blocks of length $N_k$ concatenated of blocks from some previous step $s$ of the construction, with $s$ ranging from the reference index $p$ to the index $k-1$ of the preceding step.
We will denote $\G^{(s)} = (\G_s)^{\frac{N_k}{N_s}}$ treated as a probability space with the normalized counting measure. Any block $B$ from any of these spaces decomposes as a concatenation $Q_1Q_2\dots Q_q$ of blocks from $\G_p$, with $q=\frac{N_k}{N_p}$. Fix a code $f\in\mathcal F$ (we have $r_{\!f}-1<\!\!\!<N_{\!p}$) and we fix a block $C$ of length $N_k$ appearing in $y$, ending before the position $m^2N_k$ (only such blocks appear in (R)). This block can be represented as the concatenation
$$
C=U_1I_1U_2I_2\dots U_qI_q,
$$
where each $U_i$ has length $N_{\!p}-r_{\!f}+1$ and each $I_i$ has length $r_{\!f}-1$. 
On each of the spaces $\G^{(s)}$ we define the \emph{$p$-approximate correlation function}
\begin{equation}\label{decompo}
B\mapsto\frac1q\sum_{i=1}^q\overline{f(Q_i)U_i}=:\,^{^p\!}\overline{f(B)C}.
\end{equation}
It is obvious that $\,^{^p\!}\overline{f(B)C}$ differs from $\overline{f(B)C}$ by at most $\frac{r_{\!f}-1}{N_{\!p}}$ (less than $\delta$), because this is the contribution of $\{-1,1\}$-valued terms in the evaluation of $\overline{f(B)C}$ not included in the evaluation of $\,^{^p\!}\overline{f(B)C}$. We have 
$$
\overline{f(Q_i)U_i} = \frac1{N_{\!p}-r_{\!f}+1}\sum_{l=1}^{N_{\!p}-r_{\!f}+1}(f(Q_i))_l(U_i)_l,
$$
where $(f(Q_i))_l$ and $(U_i)_l$ are single symbols in $f(Q_i)$ and $U_i$, respectively. Changing the order of summation, we get
$$
\,^{^p\!}\overline{f(B)C}= \frac1{N_{\!p}-r_{\!f}+1}\sum_{l=1}^{N_{\!p}-r_{\!f}+1}\frac1q\sum_{i=1}^q(f(Q_i))_l(U_i)_l.
$$
In evaluating the expected value, which will be denoted by $E^{(s)}(\,^{^p\!}\overline{f(\cdot)C})$, over any of the above spaces $\G^{(s)}$, the terms $(U_i)_l$ are constant, so we can write
$$
E^{(s)}(\,^{^p\!}\overline{f(\cdot)C})=
\frac1{N_{\!p}-r_{\!f}+1}\sum_{l=1}^{N_{\!p}-r_{\!f}+1}\frac1q\sum_{i=1}^q(U_i)_l\,E^{(s)}\!(f_{i,l}),
$$
where $f_{i,l}$ is the $\{-1,1\}$-valued variable $B\mapsto(f(Q_i))_l$ selecting one symbol of $f(B)$
(precisely, the symbol at the position $iN_{\!p}+l$). 

Now, on the largest space $\G^{(p)}$, i.e., on all blocks of length $N_k$ which are concatenations of blocks from $\G_p$, the variables $f_{i,l}$ with a common index $l$ have the same distribution for all indices $i$, and hence a common expected value denoted $E^{(p)}_l$ (note that $|E^{(p)}_l|\le 1$). Then
$$
E^{(p)}(\,^{^p\!}\overline{f(\cdot)C})= \frac1{N_{\!p}-r_{\!f}+1}\sum_{l=1}^{N_{\!p}-r_{\!f}+1}E^{(p)}_l\,\frac1q\sum_{i=1}^q(U_i)_l.
$$
The last average is the average of $y$ along an arithmetic progression with step $N_{\!p}$ and consisting of $q=\frac{N_k}{N_{\!p}}$ terms, contained in the first $m^2\frac{N_k}{N_{\!p}}$ terms. Since $k\ge K_m$,
and hence $\frac{N_k}{N_{\!p}}\ge\frac{N_{\!K_m}}{N_{\!p}}\ge L(\epsilon,m^2)$, the condition (a) implies that the last average (for every $l$) is less than $\epsilon$, and thus so is the double average.

Having proved that $E^{(p)}(\,^{^p\!}\overline{f(\cdot)C})<\epsilon$, we need to control how this expected value changes as we pass to smaller spaces $\G^{(s)}$, till we reach $\G^{(k-1)}$. Decomposing each $B\in\G^{(s-1)}$ into $\frac{N_k}{N_{\!s}}$ subblocks $B_i\in\G_s$ and also decomposing $C$ into corresponding subblocks $C_i$ of length $N_{\!s}$, we can write
$$
\,^{^p\!}\overline{f(B)C}=\frac{N_{\!s}}{N_k}\sum_{i=1}^{\frac{N_k}{N_{\!s}}}\,^{^p\!}\overline{f(B_i)C_i}.
$$
Since for $\tau\in\{s-1, s\}$ the blocks $B_i$ range over a set $\G_{(\tau)}$ independent of $i$ (more specifically, $\G_{(s-1)}=(\G_{s-1})^{\frac{N_{\!s}}{N_{\!s-1}}}$ and  $\G_{(s)}=\G_s$), we have
$$
E^{(\tau)}(\,^{^p\!}\overline{f(\cdot)C})=\frac{N_{\!s}}{N_k}\sum_{i=1}^{\frac{N_k}{N_{\!s}}}E_{(\tau)}(\,^{^p\!}\overline{f(\cdot)C_i}),
$$
where the latter expectation is over $\G_{(\tau)}$ (the dot represents the varying block $B_i$). Now, in the passage from $\tau=s-1$ to $\tau=s$ we must renormalize the measure from $(\G_{s-1})^{\frac{N_{\!s}}{N_{\!s-1}}}$ to $\G_s$. The expected value of any function with values in $[-1,1]$ may change by at most $2(1-\gamma_{\!s})$.\footnote{Here is the abbreviated derivation for $A\subset\Omega$ of measure $\gamma$. The function of modulus at most~1 and the probability measure are omitted: $|\int_\Omega-\frac1\gamma\int_A| = |(1-\frac1\gamma)\int_A+\int_{A^c}|\le |\frac{\gamma-1}\gamma\gamma|+(1-\gamma)$.} After averaging over $i$ we get
$$
|E^{s}(\,^{^p\!}\overline{f(\cdot)C}) - E^{(s-1)}(\,^{^p\!}\overline{f(\cdot)C})|< 2(1-\gamma_{\!s}).
$$
Composing over $s=p+1,\dots,k-1$ and using~(A), we arrive at 
\begin{equation}\label{citu}
|E^{(k-1)}(^{^p\!}\overline{f(\cdot)C})|< \epsilon+2\cdot\!\!\!\sum_{s=p+1}^{k-1}(1-\gamma_{\!s})<\epsilon+\delta.
\end{equation}
Recall that $\mathsf X$ appearing in the condition~(B) equals $\overline{f(B)C}$ restricted to $\G^{(k-1)}$, and that $\overline{f(B)C}$ differs from $^{^p\!}\overline{f(B)C}$ by less than $\delta$. Thus we can conclude the proof of~(B):
$$
|E\mathsf X| =|E^{(k-1)}(\overline{f(\cdot)C})|<\epsilon+2\delta.
$$
\medskip

We pass to the proof of~(C). We will need once again to refer to the $p$-approximate correlations $\,^{^p\!}\overline{f(B)C}$, but now on the spaces $\G_s$ ($s=p,\dots,k-1$), (the blocks $B$ will now have lengths $N_{\!s}$) and on the space $(\G_{k-1})^m$ (on which we have already been working). The code $f$ remains fixed, and $C$ is any block of the appropriate length ($N_s$ or $N_k$) appearing in $y$, ending before the position $m^2N_k$.

Suppose $\mathbf v_{\!s-1}$ is a bound on the variance of all $p$-approximate correlations on $\G_{s-1}$. Then the variance on the space $(\G_{s-1})^{\frac{N_{\!s}}{N_{\!s-1}}}$ is at most $\frac{N_{\!s-1}}{N_{\!s}}\mathbf v_{\!s-1}$, because we are averaging independent random variables. Now, $\G_s$ is a subset of the above product space, where it has measure $\gamma_s$, which, by the inductive assumption,  is larger than $\frac12$. Conditioning on such a subset can enlarge the variance at most 4 times.\footnote{To see this, write the variance as $\int\!\!\!\int\frac12(x-y)^2\,dxdy$ where $dx$ and $dy$ stand for the distribution on $\R$ of the random variable. The set on which we condition in $\R^2$ has measure larger than~$\frac14$.} Thus we obtain the estimate
$$
\mathbf v_{\!s}\le4\tfrac{N_{\!s-1}}{N_{\!s}}\mathbf v_{\!s-1}.
$$
By recursive application of the above, we can estimate $\mathbf v_{k-1}$ referring to the step $p$ and safely estimating the variances on $\G_p$ by $\mathbf v_{\!p}=2$ (our variables take values in $[-1,1]$):
$$
\mathbf v_{\!k-1}\le 4^{k-1-p}\tfrac{N_{\!p}}{N_{\!k-1}}\cdot 2\le 2N_{\!p}\tfrac{4^{k-1}}{N_{\!k-1}}.
$$
\smallskip


Denote by $\overline{\mathsf X}$ (to match the notation in \eqref{Hoeffding}), the $p$-approximate correlation regarded on the product space $(\G_{k-1})^m$ (on which it is indeed the average of $m$ independent random variables). Now \eqref{Hoeffding} applies and reads:
$$
\P\{|\overline{\mathsf X}-E\overline{\mathsf X}|\geq \epsilon\}\le 2\cdot4^m{\mathbf v_{k-1}}^{\frac\epsilon2m}\le 2\cdot4^m(2N_{\!p})^{\frac32}(\tfrac{4^{k-1}}{N_{\!k-1}})^{\frac32}
$$
(we have also used the equality $\epsilon=\frac3m$).
\medskip
The expectation $E\overline{\mathsf X}$ coincides with what was previously denoted by $E^{(k-1)}(^{^p\!}\overline{f(\cdot)C})$, so, by the already proved inequality \eqref{citu}, we have
$|E\overline{\mathsf X}|<\epsilon+\delta$. We can thus continue:
$$
\P\{|\overline{\mathsf X}-E\overline{\mathsf X}|\geq \epsilon\}\ge\P\{|\overline{\mathsf X}|\ge2\epsilon+\delta\}.
$$
Recall that $\overline{\mathsf X}$ differs from the corresponding correlation function $B\mapsto\overline{f(B)C}$ by less than $\delta$. This implies that
$$
\P\{|\overline{f(B)C}|\geq 2(\epsilon+\delta)\}\ \le\ 
2\cdot4^m(2N_{\!p})^{\frac32}(\tfrac{4^{k-1}}{N_{\!k-1}})^{\frac32}.
$$
The above concerns the probability on $(\G_{k-1})^m$, a fixed block $C$ in $y$ and a fixed code $f\in\mathcal F$. The condition (R) requires the inequality $|\overline{f(B)C}|\ge2(\epsilon+\delta)$ to be satisfied for $(m^2-1)N_{\!k}\le m^3N_{\!k-1}$ blocks $C$ and all codes $f\in\mathcal F$. Since $\#\mathcal F\le m$, the overall probability $1-\gamma_k$ of a block $B$ failing (R) is estimated by 
$$
m^4N_{\!k-1}\cdot2\cdot4^m(2N_{\!p})^{\frac32}(\tfrac{4^{k-1}}{N_{\!k-1}})^{\frac32}=
\alpha(m)\tfrac{8^{k-1}}{\sqrt{N_{\!k-1}}}\le\alpha(m)\bigr(\tfrac8{\sqrt M}\bigr)^{k-1}\le
\alpha(m)\bigr(\tfrac89\bigr)^{k-1}
$$
(recall that $M\ge81$). This ends the proof of~(C) and thus of the lemma.
\end{proof}

Since we have proved, in particular, that all $\gamma_k$ are larger than $\frac12$, it now becomes certain that the entropy of $\Sigma$ can be made arbitrarily close to $\log N$. It remains to prove lack of correlation between $\Sigma$ and $y$. Let $f$ be any $\{-1,1\}$-valued function depending on finitely many nonnegative coordinates. Fix some point $x\in\Sigma$ and pick $n\in\N$. Let $k$ be the smallest integer such that $n< m^2N_k$. If $f$ is not in $\mathcal F=\mathcal F_k$ then we simply must pick a larger $n$. So, we can assume that $f\in\mathcal F$. Now, $x\in\Sigma_k$, which means that $x_1^n$ is a concatenation of the blocks from $\G_k$, except that the first and last component blocks may be incomplete. The contribution of these parts in the length is at most $\frac{2N_k}n$, and since $n\ge m_{k-1}^2N_{k-1}\ge (m-1)^2N_{k-1}>(m-2)N_k$, this contribution is less than $\frac2{m-2}$, and such is also the maximal contribution of these parts in the evaluation of the correlation between $x_1^n$ and $y_1^n$. The rest of the correlation is the average of the correlations of the complete component blocks from $\G_k$ with their respective subblocks of length $N_k$ of $y$. Since all these subblocks end before the position $m^2N_k$, by (R), each of these correlations is less than $2(\epsilon+\delta)$ in absolute value. Jointly, the absolute value of the correlation of $x_1^n$ with $y_1^n$ does not exceed 
$$
\tfrac 2{m-2}\cdot1 + \tfrac{m-4}{m-2}\cdot2(\epsilon+\delta).
$$
Obviously, as $n$ grows, so does $k$, and so does $m$, while both $\epsilon$ and $\delta$ tend to zero. This proves the desired uncorrelation condition concluding the entire proof of the main result.
\end{proof}

\end{document}